\documentclass[11pt]{amsart}

\usepackage{graphicx,psfrag,mathrsfs}

\usepackage{latexsym}
\usepackage{amssymb,amsfonts,amsmath,amsthm}
\usepackage{graphicx}
\usepackage{enumerate}
\usepackage{comment}
\usepackage{fullpage}

\newtheorem{theorem}{Theorem}  

\newtheorem{lemma}[theorem]{Lemma}
\newtheorem{observation}[theorem]{Observation}

\newtheorem{corollary}[theorem]{Corollary}

\newtheorem{conjecture}[theorem]{Conjecture}

\usepackage[normalem]{ulem}
\usepackage{xcolor}

\renewcommand{\epsilon}{\varepsilon}

\numberwithin{theorem}{section}
\numberwithin{equation}{section}
\def\ch{{\rm ch}}
\def\pr{{\rm Pr}}
\title{Dynamic choosability of triangle-free graphs and sparse random graphs}

\author{Jaehoon Kim and Seongmin Ok}

\thanks{The research of the first author was partially supported by the European Research Council under the European Union's Seventh Framework Programme (FP/2007--2013) / ERC Grant Agreements no. 306349 (J.~Kim).\\
The research of the second author is partially supported by Basic Science Research Program through the National Research Foundation of Korea(NRF) funded by the Ministry of Education(NRF-2016R1D1A1B03932398).}

\begin{document}

\begin{abstract}
 The \textit{$r$-dynamic choosability} of a graph $G$, written $\ch_r(G)$, is the least $k$ such that whenever each vertex is assigned a list of at least $k$ colors a proper coloring can be chosen from the lists so that every vertex $v$ has at least $\min\{d_G(v),r\}$ neighbors of distinct colors. Let $\ch(G)$ denote the choice number of $G$. In this paper, we prove $\ch_r(G)\leq (1+o(1))\ch(G)$ when $\frac{\Delta(G)}{\delta(G)}$ is bounded. 
 We also show that there exists a constant $C$ such that the random graph $G=G(n,p)$ with $\frac{6\log(n)}{n} < p \leq \frac{1}{2}$ almost surely satisfies $\ch_2(G) \leq \ch(G)+C$. Also if $G$ is a triangle-free regular graph, then we have $\ch_2(G)\leq \ch(G)+86$. \end{abstract}

\date{\today}

\maketitle

\section{Introduction} \label{Introduction}

The \emph{Secret Sharing Scheme} is a method to distribute an important \emph{key} or \emph{secret} amongst a group of people, each of whom is assigned a share of the data. Retrieving the full information requires certain number of distinct shares, called the \emph{threshold}, to be collected. Let us say the proximity between the participants are modeled by a graph. To ensure quick access to the key whenever needed, we want the retrieval be achievable in the neighborhood of each vertex. The problem of finding how many distinct shares in total are necessary can be expressed in terms of the $r$-dynamic coloring defined below.

In a communication network, two adjacent computers must be assigned different resources. To make many resources accessible, each computer need to be able to obtain many resources amongst its neighbors.  However, expecting all neighbors of each computer to have distinct resources demands too many types of resources.
Alternatively, we can specify a threshold $r$ such that for a computer with $d$ neighbors must have access to at least $\min\{r,d\}$ distinct types of resources in its neighbors.  We ask how many types of resources are required in this model by considering $r$-dynamic coloring defined below.

For $k\in \mathbb{N}$, we denote $[k]:= \{1,\dots,k\}$. For a graph $G$ and a vertex $v\in V(G)$, we denote $N_G(v)$ to be the set of neighbors of $v$ in $G$ and $d_G(v):= |N_G(v)|.$ Let $\Delta(G) := \max_{v\in V(G)}\{ d_{G}(v)\}, \delta(G):=\min_{v\in V(G)}\{d_{G}(v)\}$. For $d\in \mathbb{N}$, we say a graph is $d$-regular if $\Delta(G)=\delta(G)=d$. In this paper, $\log$ denote the natural logarithm.

A \emph{proper $k$-coloring} of a graph $G$ is an assignment $f : V(G) \longrightarrow C$ with a set $C$ of $k$ colors such that $f(u), f(v)$ are different whenever $u,v$ are adjacent in $G$. The smallest number $k$ such that a proper $k$-coloring of $G$ exists is called the \emph{chromatic number} of $G$, denoted by $\chi(G)$, which is one of the most important graph parameters.

The concept of the dynamic coloring was first introduced by Montgomery \cite{M}. 
For given $r, k\in \mathbb{N}$, an \emph{$r$-dynamic $k$-coloring} $f$ of a graph $G$ is a proper $k$-coloring with the additional assumption that $|f(N(v))| \geq \min \{r, d(v)\}$ for each vertex $v$. 

The \emph{$r$-dynamic chromatic number} of $G$, denoted by $\chi _r(G)$, is defined as the smallest number $k$ such that an $r$-dynamic $k$-coloring of $G$ exists. A \emph{dynamic coloring} and the \emph{dynamic chromatic number} of a graph $G$ refer to a \emph{2-dynamic coloring} and the \emph{2-dynamic chromatic number} of $G$, respectively.

For a given graph $G$, the \emph{square} $G^2$ of $G$ denote the graph obtained from $G$ by adding all edges joining two nonadjacent vertices sharing a common neighbor. Coloring the square of a graph has been extensively studied both combinatorially and algorithmically with applications in communication network. One of the motivations of $r$-dynamic coloring is that it gives a spectrum between chromatic numbers $\chi(G)$ and $\chi(G^2)$ because of the following observation:
$$\chi (G) \leq \chi_2(G) \leq \chi_3(G) \leq \cdots \leq \chi_{\Delta(G)} (G) = \chi(G^2).$$

In \cite{M}, Montgomery conjectured the following:

{\conjecture \label{conj:montgomery} If $G$ is a regular graph, then $\chi_2(G) \leq \chi(G) +2$.}

Conjecture \ref{conj:montgomery} was proven for several classes of regular graphs, such as bipartite regular graphs \cite{AGJ1}, claw-free graphs \cite{M} and graphs with diameter at most $2$ and chromatic number at least $4$ \cite{A12}. For every $d$-regular graph $G$, Alishahi \cite{A12} provided an upper bound with additional logarithmic term:  $\chi_2(G) \leq \chi(G) + 14.06 \log d + 1$. Ahadi et al.~\cite{AADG} posed the following conjecture, which generalizes Conjecture \ref{conj:montgomery}.

{\conjecture \label{conj:AADG} If $G$ has maximum degree $\Delta$ and minimum degree $\delta$, then $\chi_2(G) \leq \chi(G) + \lceil \frac{\Delta}{\delta} \rceil+1$.}

Recently, Bowler et al. \cite{BELMPS} found a counterexample for Conjecture \ref{conj:montgomery} by constructing a $d$-regular graph with $\chi_r(G)=r \chi(G)$ for large $d$. In this paper, we consider list-coloring-variation of Conjecture~\ref{conj:AADG} for certain classes of graphs.

The \emph{$r$-dynamic choosability} of a graph $G$, denoted by $\ch _r(G)$, is the least positive integer $k$ such that the following holds: 

for any given sets $L(v)$ for each vertex $v\in V(G)$ with $|L(v)| \geq k$, there exists an $r$-dynamic coloring $f : V(G) \rightarrow \bigcup_{v \in V(G)} L(v)$ such that $f(v) \in L(v)$ for all $v\in V(G)$.

Akbari et al.~\cite{AGJ2} proved $\ch_2(G)\le\Delta(G)+1$ when
$\Delta(G)\geq 3$ and no component is the $5$-cycle $C_5$. Kim and Park~\cite{KP} proved $\ch_2(G)\le4$ if $G$ is planar with girth at least $7$, and $\ch_2(G)\le k$ if $k\ge4$ and $\Delta(G)  \leq \frac{4k}{k+2}$. Kim et al. \cite{KLP} proved $\chi_2(G)\le 4$ if $G$ is planar and no component of $G$ is $C_5$; also, they proved $\ch_2(G)\le 5$ if $G$ is planar.

Similarly to Conjecture \ref{conj:montgomery}, Akbari et al. \cite{AGJ2} conjectured that $\ch_2(G) \leq \max\{\ch(G), \chi_2(G)\}$. However, Esperet \cite{E} constructed a graph $G_k$ for each $k \geq 3$ such that $\chi_2(G_k) = \ch(G_k) = 3$ whereas $\ch_2(G_k) \geq k$, thereby disproving Akbari's conjecture. All of Esperet's examples are 2-degenerate and $\ch(G_k) = 3$. In Section \ref{construction}, we construct a more robust class of counterexamples of the conjecture. To be specific, we construct graphs $G$ with small $\chi(G)$ and $\ch(G)$ such that all the value of $\ch_r(G)-\chi_r(G)$, $\chi_r(G)-\ch(G)$ and $\delta(G)$ are all arbitrarily large. These graphs show that finding a good bound on $\ch_r(G)$ using the values of $\chi(G), \ch(G)$ and $\chi_r(G)$ for general graph $G$ may be difficult. However, our graphs is far from regular as $\Delta(G)-\delta(G)$ is big. 

Our main result, Theorem \ref{thm:main}, proves an upper bound on the $r$-dynamic choosability of almost regular graphs in terms of $\ch(G)$ and $\frac{\Delta(G)}{\delta(G)}$. 
\begin{theorem} \label{thm:main}
Suppose $r,s,\ell \in \mathbb{N}$ with $s\geq r-1$ and $r\geq 2$. Suppose that $G$ is a graph with $\Delta(G)=\Delta$, $\delta(G)=\delta$, $\ch(G) = \ell$. If 
$$\left((r+1)\log(\Delta) + (r-1)\log(r) +1\right) \left(\frac{\ell +s}{s}\right)^{r-1}\leq \delta,$$ then we have
 $$\ch_r(G) \leq \ch(G) + s + r-2.$$
\end{theorem}
One implication of Theorem \ref{thm:main} is Corollary \ref{random graphs}: there exists a constant $C$ such that almost all $n$-vertex graphs with average degree at least $6\log(n)$ satisfies $\ch_2(G)\leq \ch(G)+C$. 
Note that if there is an edge in $N_{G}(v)$, then $N_{G}(v)$ is non-monochromatic in any proper coloring of $G$. Thus $\ch(G(n,p)) = \ch_2(G(n,p))$ almost surely holds for $p\gg n^{-\frac{2}{3}}$ since neighborhoods of all vertices contain an edge in $G(n,p)$. On the other hand, for $p\ll n^{-\frac{2}{3}}$, almost surely the neighborhoods of all vertices are independent, so the problem is not trivial. However, Corollary \ref{random graphs} gives an upper bound for the dynamic choosability of $G(n,p)$ for all $p\geq \frac{6\log(n)}{n}$.

Another implication of Theorem \ref{thm:main} is Theorem \ref{triangle-free}: if $G$ is a triangle-free graph with $\delta(G) \geq 6\log(\Delta(G)) + 2$, then $\ch_2(G) \leq \ch(G)+\frac{86\Delta(G)}{\delta(G)}$. As mentioned before, the vertices with independent neighborhoods are dangerous for dynamic coloring, so that the triangle-free graphs are difficult to color dynamically. However, Theorem \ref{triangle-free} gives us a good upper bound for triangle-free graphs in a form similar to Conjecture \ref{conj:AADG}. \bigskip\bigskip\bigskip

\section{Graphs of large $r$-dynamic choosability} \label{construction}

We shall start with introducing the notion of $r$-strong coloring of a hypergraph.
Let $G$ be a graph and let $H(G)$ be the hypergraph on $V(G)$ with the edge set $\{ N(v):v \in V(G)\}$. We say that a vertex coloring of $H(G)$ is \emph{$r$-strong} if each edge $e$ of $H$ have at least $\min\{r,|e|\}$ distinct colors. We define $\chi^r(H)$ be the least $k$ such that there exists an $r$-strong $k$-coloring of $H$. The $r$-strong choosability of $H$, which we denote by $\ch^r(H)$, is the least $k$ such that an $r$-strong coloring can be chosen from the lists whenever each vertex of $H$ is assigned a list of at least $k$ colors. The \emph{incidence graph} of a hypergraph $H$ is the graph $G$ with $V(G) = V(H)\cup E(H)$ and $E(G)=\{ve: v\in e, v\in V(H), e\in E(H)\}$. The following observation is obvious.

\begin{observation} \label{incidence}
If $G$ is the incidence graph of a hypergraph $H$, then $\chi_r(G) \geq \chi^r(H)$, and $\ch_r(G)\geq \ch^r(H)$.
\end{observation}

\begin{theorem} \label{example}
For $m,k,r \in \mathbb{N}$ with $k \geq r \geq 2$, there exists a bipartite graph $G$ with 
$$\ch_r(G)-\chi_r(G)\geq m,\enspace \chi_r(G)-\ch(G) \geq m, \enspace \text{and} \enspace \ch(G) \leq k+1.$$
\end{theorem}

\begin{proof}
Consider a $(k-r+2)$-uniform hypergraph $H$ with $\ch(H)-\chi(H) \geq m+r^2+2r-2$ and $\chi(H) > m+k$. It is well known that such a hypergraph exists. For example, see \cite{PV} to check that a large complete $(m+2k)$-partite $(k-r+2)$-uniform hypergraph suffices.

We take a set of $r-2$ vertices $X$ disjoint from $V(H)$, and replace every edge $e$ of $H$ with $e\cup X$ to get $H'$. It is easy to check $\ch^r(H')-\chi^r(H') \geq \ch(H) - (\chi(H)+r-2) \geq m+r^2+r$. We may assume that the number of vertices in $H'$ is a multiple of $k$ by adding some isolated vertices.
Now we add $r$ disjoint perfect matchings $M_1,M_2,\cdots, M_{r}$ to $H'$, then this may increase $\chi^r(H')$ by at most $r^2$, so we still have $\ch^r(H')-\chi^r(H')\geq m+r$. Let $G$ be the incidence graph of $H'$, and let $\{A,B\}$ be the bipartition of $G$ such that $A=V(H')$, $B=E(H')$.

Since $G$ is $k$-degenerate, $\ch(G)\leq k+1$. By Observation \ref{incidence}, $\chi_r(G) \geq \chi^r(H') >m+k$. We shall take an $r$-strong $\chi^r(H')$-coloring $f$ of $H'$, and let $\alpha_1,\alpha_2,\cdots,\alpha_r$ be new colors not used by $f$. We define a coloring $g$ of $G$ as follows. 

\begin{displaymath}
g(v):=\left\{
\begin{array}{ll}
f(v) & \mbox{ if $v\in A$} \\
\alpha_i & \mbox{ if $v\in M_i \cap B$ for some $i$, $1 \leq i \leq r$} \\
\alpha_r & \mbox{ otherwise}
\end{array}
 \right.
 \end{displaymath}

Since the colors used by $B$ are not used by $A$, the coloring $g$ is a proper coloring. The neighborhood of each vertex in $B$ contains at least $r$ vertices of different colors because $f$ is an $r$-strong coloring. Also, the neighborhood of each vertex in $A$ contains vertices with colors $\alpha_1,\alpha_2,\cdots, \alpha_r$ since it is covered by each of $M_1,M_2,\cdots, M_r$. Thus $g$ is an $r$-dynamic coloring of $G$, and 
$\chi_r(G)\leq \chi^r(H')+r$.

However, we have $\ch_r(G)\geq \ch^r(H)$ by Observation \ref{incidence}.
Thus we conclude 
$$\ch_r(G)-\chi_r(G)\geq m, \enspace \chi_r(G)-\ch(G) \geq m, \enspace \text{and} \enspace \ch(G) \leq k+1.$$
\end{proof}

\section{Proof of Theorem \ref{thm:main}}

Before proving Theorem \ref{thm:main}, we introduce the following simple lemma. For a hypergraph $H$, we say a set $T$ in $V(H)$ is a \emph{transversal} of $H$ if $T$ intersects all edges of $H$. If $T$ is a transversal of $H$ with $|T|=r$, then we say it is an $r$-\emph{transversal}. For a hypergraph $H$, let $\tau(H)$ be the minimum size of a transversal of $H$. Note that for $r'\in [n]$ the number of $r'$-transversals of an $n$-vertex hypergraph $H$ is zero if and only if $r' < \tau(H)$.

\begin{lemma}\label{transversal}
For $r\in \mathbb{N}$ and a $k$-uniform hypergraph $H$, there are at most $k^r$ distinct $r$-transversals of $H$.
\end{lemma}
\begin{proof}
We use induction on $r$.  Take a smallest natural number $r$ such that the lemma doesn't hold. We may assume that $E(H) \neq \emptyset$.

If $r=1$, then we choose an arbitrary edge $e \in E(H)$. Then any $1$-transversal $T$ must satisfy $T\subseteq e$ and $|T|=1$. Thus $H$ has at most $k$ distinct $1$-transversals.

Assume $r>1$. Let $e=\{v_1,v_2,\cdots, v_k\} \in E(H)$. For each $i\in [k]$, we consider the hypergraphs $H_i = H-v_i$.
By the induction hypothesis, for each $i\in [k]$, the number of $(r-1)$-transversals of $H_i$ is at most $t_i \leq k^{r-1}$.
Note that if $A$ is an $r$-transversal of $H$ and $v_i \in A\cap e$, then $A-v_i$ is an $(r-1)$-transversal of $H_i$.
Thus the number of $r$-transversals of $H$ is at most $\sum_{i=1}^{k} t_i \leq k \cdot k^{r-1} \leq k^{r}$. It is a contradiction to the choice of $r$, thus the lemma holds.
\end{proof}

We also know the following simple bound by a slight modification of Theorem 2.1 from \cite{JKOW}.
\begin{theorem}\label{rd}
If a graph $G$ has maximum degree $\Delta$, then 
$$\ch_r(G) \leq r\Delta+1.$$
\end{theorem}

Now we prove our main result, Theorem \ref{thm:main}.

\begin{proof}[Proof of Theorem \ref{thm:main}]
The theorem is trivial from Theorem \ref{rd} if $\ell+s+r-2\geq r\Delta +1.$ Thus we may assume $\ell+s+r-2 \leq r\Delta$. Consider a list assignment $L_v$ of $\ell+s+r-2$ colors for every vertex $v \in V(G)$. 
For each $v\in V(G)$, we choose a sub-list $L'_v \in \binom{L_v}{\ell}$ uniformly at random. We consider $H$, the neighborhood hypergraph of $G$, defined as $$V(H)=V(G),\enspace E(H) = \{N_G(v): v\in V(G)\}.$$ 
Note that each edge of $H$ intersects at most $\Delta^2$ other edges since $\Delta(G)\leq \Delta$.

We shall use the Lov\'{a}sz Local Lemma to show that there is a choice of sublists $L'_v$ such that each list coloring $f$ with $f(v) \in L'_v$ yields an $r$-strong coloring of $H$. Since $|L'_v| = \ell = \ch (G)$ for all $v\in V(G)$, we can find a proper coloring of $G$ from this list assignment, which becomes automatically an $r$-dynamic list coloring of $G$.

For each $v\in V(G)$, let $H_v$ and $H'_v$ be the hypergraphs defined on the colors such that 
$$V(H_v)= \bigcup_{w\in N_{G}(v)} L_w \enspace\text{and} \enspace E(H_v) = \{ L_{w} : w\in N_G(v) \},$$ 
$$V(H'_v)= \bigcup_{w\in N_{G}(v)} L'_w\enspace \text{and} \enspace E(H'_v) = \{ L'_{w}: w \in N_G(v)\}.$$  We consider the following event $A_v$.
\begin{equation}\label{event}
A_v : \tau(H'_v) \leq r-1.
\end{equation}
Let us estimate the probability $\pr (A_v)$. Fix a vertex $v$ and let $N_G(v)=\{ u_1,u_2,\cdots, u_{d_G(v)}\}$. For an $(r-1)$-transversal $P$ of $H_v$ and $i\in [d_G(v)]$, the probability that $L'_{u_i}$ intersect $P$ is at most
$$1- \frac{{{\ell+s+r-2-|P\cap L_{u_i}|}\choose{\ell}}}{{{\ell+s+r-2}\choose {\ell}}}\leq 1- \frac{{{\ell+s-1}\choose{\ell}}}{{{\ell+s+r-2}\choose {\ell}}}.$$
Let $B_v(P)$ be the event that $P$ is a transversal of $H'_v$. In other words, 
$P$ intersects every edge of $H'_v$.
Since the choice of $L'_u$ and the choice of $L'_w$ are independent for two distinct vertices $u,w \in N_G(v)$, we have
$$ \pr (B_v(P)) \leq \prod_{i=1}^{d_G(v)} \left(1-\frac{{{\ell+s-1}\choose{\ell}}}{{{\ell+s+r-2}\choose {\ell}}} \right) \leq \left(1-(\frac{s}{\ell+s})^{r-1}\right)^\delta < e^{-\delta(\frac{s}{\ell+s})^{r-1}} $$

Let $\mathbf{T}$ be the set of all $(r-1)$-transversals of $H_v$, then by Lemma \ref{transversal}, $|\mathbf{T}|\leq (\ell+s+r-2)^{r-1}$. Since every edge of $H'_v$ is a subset of an edge of $H_v$, every $(r-1)$-transversal of $H'_v$ also belongs to $\mathbf{T}$. Note that $\tau(H'_v) \leq r-1$ if and only if there exists a set $P\in \mathbf{T}$ which is a transversal of $H'_v$.  Thus for all $v\in V(G)$,
$$\pr (A_v) \leq \sum_{P\in \mathbf{T}} \pr (B_v(P)) < (\ell+s+r-2)^{r-1} e^{-\delta(\frac{s}{\ell+s})^{r-1}} $$ 
However, events $A_v$ is mutually independent of the set $\{ A_u: N_{G}(v)\cap N_{G}(u)  = \emptyset\}$ of events.
So each event $A_v$ is mutually independent of all but at most $\Delta^2$ of other events $A_u$.  Since $ ((r+1)\log(\Delta) + (r-1)\log (r) +1) (\frac{\ell+s}{s})^{r-1}\leq \delta$, for each $v\in V(G)$ we have 
$$e\Delta^2 \pr (A_v) < e\Delta^2(\ell+s+r-2)^{r-1} e^{-\delta(\frac{s}{\ell+s})^{r-1}} \leq \frac{ e\Delta^2(\ell+s+r-2)^{r-1}}{ e r^{r-1}\Delta^{r+1}} \leq  1.$$
By the Lov\'{a}sz Local Lemma, there is a choice of sublists $L'_v$ which avoids all the events $A_v$ simultaneously. 
Since $|L'_v| = \ch(G)$ for every $v\in V(G)$, there exists a proper coloring $f$ of $G$ such that $f(v) \in L'_v$. Moreover, $f(N_G(v))$ is a transversal of $H'_v$, thus  $|f(N_G(v))| \geq r$ as $\tau(H'_v) \geq r$. Therefore, the coloring $f$ is in fact a $r$-dynamic list coloring of $G$, and $\ch_r(G)\leq \ell+s+r-2$.\end{proof}

\section{Consequences of Theorem \ref{thm:main}}

Theorem \ref{thm:main} immediately implies the following corollary, which shows that $\ch_r(G) \leq (1+o(1))\ch(G)$ for graphs $G$ with bounded $\frac{\Delta(G)}{\delta(G)}$.
\begin{corollary}\label{upper bound}
For given $\epsilon>0$, $r \in \mathbb{N}\setminus\{1\}$ and $k\geq 1$, there exists $\ell_0= \ell_0(r,k,\epsilon)$ such that the following holds.
Suppose that $G$ is a graph with $\Delta(G)=\Delta$, $\delta(G) =\delta$ and $\frac{\Delta}{\delta} \leq k$.
If $\ch(G) \geq \ell_0$, then 
$$\ch_r(G) \leq (1+\epsilon) \ch(G) .$$
\end{corollary}
\begin{proof}
Let $\ch(G)=\ell$. It is enough to prove that if $\ell \geq 6^{2r}r^{3r}k^2$  then
$$ \ch_r(G) \leq \ell + \lceil (3kr \ell^{r-2} \log(\ell))^{\frac{1}{r-1}} \rceil +r-2.$$
If $G$ is an odd cycle or a clique, then the inequality is trivial. So we may assume $\ell \leq \Delta$. Now we choose $s = \lceil (3kr \ell^{r-2} \log(\ell) )^{\frac{1}{r-1}}\rceil$. 
Since $\ell \geq 6^{2r} r^{3r} k^2$, we have $\frac{\ell}{\log(\ell)} \geq 6^{r} r^{2r-1} k$ which implies $\frac{r^2 s}{\ell} \leq \frac{1}{3}$. We need to verify the condition $((r+1)\log( \Delta) +(r-1)\log(r)+ 1)(\frac{\ell+s}{s})^{r-1} \leq \delta$ to apply Theorem \ref{thm:main}. Note that $$((r+1)\log(\Delta) +(r-1)\log(r)+ 1) \leq (1+\frac{1}{r+1})(r+1)\log(\Delta) \leq \frac{4}{3}(r+1)\log(\Delta) ,$$ 
since $\Delta \geq \ell \geq r^{3r}$. We also have
$$(\frac{\ell+s}{s})^{r-1} \leq (1+ \frac{rs}{\ell} + \frac{r^2s^2}{\ell^2} + \cdots + \frac{r^{r-1}s^{r-1}}{\ell^{r-1}}) (\frac{\ell}{s})^{r-1} \leq (1+\frac{r^2 s}{\ell})(\frac{\ell}{s})^{r-1} \leq \frac{4}{3}(\frac{\ell}{s})^{r-1}. $$
Hence, it is enough to show
$$ \frac{16}{9} (r+1) \log(\Delta) (\frac{\ell}{s})^{r-1} \leq \frac{\Delta}{k} \leq \delta,$$
which follows from
$$ 3kr (\frac{\ell}{s})^{r-1}\leq \frac{\ell}{\log(\ell)} \leq \frac{\Delta}{\log(\Delta)}.$$
Thus we can apply Theorem \ref{thm:main} to conclude the theorem.
\end{proof}

\begin{corollary}\label{cor: 4.2}
Let $G$ be a graph with $\Delta(G)=\Delta$ and $\delta(G)=\delta$. 
If $((r+1)\log(\Delta) + (r-1)\log(r)+ 1) (\ch(G)+1)^{r-1} \leq \delta$, then
$$\ch_r(G) \leq \ch(G)+ r-1.$$
In particular, if $r=2$, then $ (3\log(\Delta) + 2) (\ch(G)+1) \leq \delta $ implies $\ch_2(G)\leq \ch(G)+1$.
\end{corollary}

In \cite{A}, Alon mentioned Kahn's proof that shows almost surely $\ch(G(n,\frac{1}{2})) = (1+o(1))\frac{n}{2\log(n)}$.
Also, in \cite{AKS}, Alon, Krivelevich and Sudakov showed that there exists an absolute constant $c_1,c_2$ such that the random graph $G(n,p)$ almost surely satisfies $c_1\frac{np}{\log(np)}\leq \ch(G(n,p))\leq c_2\frac{np}{\log(np)}$ for all $\frac{2}{n} < p \leq \frac{1}{2}$. Since $G(n,p)$ with $p > \frac{6 \log(n)}{n}$ almost surely satisfies $\frac{\Delta(G(n,p))}{\delta(G(n,p))} \leq 12$,
these facts combined with Theorem \ref{thm:main} show the following.

\begin{corollary}\label{random graphs}
There exists an absolute constant $C$ such that for any $ \frac{6\log(n)}{n}< p =p(n) \leq 1$, $G=G(n,p)$ almost surely satisfies 
$ \ch_2(G) \leq \ch(G)+C.$\end{corollary}
\begin{proof}
Note that if $p>1/2$, then with high probability, $N_{G}(v)$ contains an edge  for all $v\in V(G)$. Thus $\ch_2(G) = \ch(G)$ almost surely holds. So, we may assume that $p\leq 1/2$.

Simple calculation with Chernoff bound shows that the random graph $G=G(n,p)$ almost surely satisfies $np - \sqrt{ 4 pn \log(n) } \leq \delta(G) \leq \Delta(G) \leq np + \sqrt{4 pn \log(n)}$. 
Thus, for $p> \frac{6\log(n)}{n}$, we have
$$ np/6 \leq  \delta(G)\leq \Delta(G)\leq 2np.$$
So we almost surely have $$ \frac{np}{ 30 \log(np)} \leq \frac{\delta(G)}{4 \log(\Delta)}.$$ Let $s = \lceil 40c_2 \rceil$ where $c_2$ is the constant in \cite{AKS} (see paragraph preceding Corollary~\ref{random graphs}) such that $\ch(G) \leq \frac{c_2 np }{\log(np)}$, so that almost surely we have
$$\frac{\ch(G) +s}{s} = 1 + \frac{\ch(G) }{s} \leq \frac{np}{30\log(np)} \leq \frac{\delta(G)}{3 \log(\Delta(G)) +2} .$$
Thus 
$$\left(3\log(\Delta) + \log(2) +1\right) \left(\frac{\ch(G) +s}{s}\right) \leq \delta.$$
Hence, Theorem \ref{thm:main} implies that almost surely we have that $\ch_2(G)\leq \ch(G)+s$.
\end{proof}

In \cite{J}, Johansson proved $\ch(G)\leq \frac{9\Delta(G)}{\log_2(\Delta(G))}\leq  \frac{13\Delta(G)}{\log(\Delta(G))}$ for all triangle-free graph $G$. Johansson's result combined with our Theorem \ref{thm:main} shows the following.

\begin{theorem}\label{triangle-free}
Let $G$ be a triangle-free graph with $\Delta(G)=\Delta$ and $\delta(G)=\delta$. If $\delta \geq 6\log(\Delta)+2$, then 
$$\ch_2(G) \leq \ch(G)+ \frac{86\Delta}{\delta}.$$
In particular, if $G$ is a regular graph, then $\ch_2(G)\leq \ch(G)+86$.
\end{theorem}
\begin{proof}
Since $\ch_2(G)\leq 2\Delta(G)+1$ by Theorem \ref{rd}, we may assume $\Delta \geq \delta \geq 43$. Thus every regular graph in our consideration satisfies $\delta\geq 9\log(\Delta)+6$. Let $s= \frac{86\Delta}{\delta}$ and we apply Theorem \ref{thm:main}. Note that $13 \leq 4 \log (43)$.

We only have to check $(3\log(\Delta) + 2)\frac{\ch(G)+s}{s} \leq \delta$. Since $\ch(G)\leq \frac{13\Delta}{\log(\Delta)}$,
$$(3 \log(\Delta) + 2) (1 + \frac{\ch(G)}{s}) \leq (3 \log(\Delta) + 2) + (3 \log(\Delta)+2 ) \frac{13\delta}{86 \log(\Delta)} \leq \frac{\delta}{3} + \left( \frac{39 \log(\Delta) + 26}{86\log(\Delta)}\right) \delta \leq \delta.$$
Thus by Theorem \ref{thm:main}, $\ch_2(G)\leq \ch(G) + \frac{86\Delta}{\delta}$.\end{proof}

More generally, Vu \cite{V} proved that there exists a positive constant $K$ such that for any graph $G$, if $G[N_G(v)]$ contains at most $\frac{\Delta^2}{f}$ edges for all $v \in V(G)$, then $\ch(G) \leq \frac{K\Delta}{\log(f)}$. By using this as in the proof of Theorem \ref{triangle-free}, we get the following.

\begin{corollary}
Let $G$ be a graph with $\Delta(G)=\Delta, \delta(G)=\delta>0$. Then, there exists a constant $K'$ satisfying the following.
If for each $v\in V(G)$, the neighborhood $G[N_G(v)]$ contains at most $\frac{\Delta^2}{f}$ edges, then 
$$\ch_2(G) \leq \ch(G)+ \frac{K'\Delta \log(\Delta) }{\delta\log(f)}.$$
\end{corollary}

\medskip

{\footnotesize \obeylines \parindent=0pt

Jaehoon Kim, School of Mathematics, University of Birmingham, Edgbaston, Birmingham, B15 2TT, UK \\ {\rm{kimJS@bham.ac.uk}} }

{\footnotesize \obeylines \parindent=0pt

Seongmin Ok, School of Computational Sciences, Korea Institute for Advanced Study, 02455 Seoul, Korea \\ seong@kias.re.kr
}
\end{document}